\documentclass[a4paper,12pt]{amsart} 
\author{David Holmes}

\usepackage{amsmath, amssymb, amsfonts, url,mathrsfs, tikz, wrapfig, newclude, amsthm}
\usepackage[all]{xy}
\def\p{\mathbb{P}}

\def\oo{\mathcal{O}}
\def\qq{\mathbb{Q}}

\def\ccc{\mathbb{C}}

\def\uv{_{\nu}}

\def\defeq{\stackrel{\text{\tiny def}}{=}}

\DeclareMathOperator{\Div}{Div}
\DeclareMathOperator{\divisor}{div}
\DeclareMathOperator{\Spec}{Spec}

\DeclareMathOperator{\ord}{ord}
\DeclareMathOperator{\res}{res}

\DeclareMathOperator{\supp}{Supp}
\DeclareMathOperator{\N}{N}
\DeclareMathOperator{\im}{Im}

\DeclareMathOperator{\Jac}{Jac}

\DeclareMathOperator{\inv}{inv}

\newcommand{\linequiv}{\sim_{\text{lin}}}
\newcommand{\cc}{\mathscr{C}}

\newcommand{\intersect}[3][\nu]{\iota_{#1}\left(#2,#3\right)}

\newcommand{\abs}[1]{\lvert#1\rvert}

\newcommand{\listdef}[3][1]{#2_{#1}, \ldots , #2_{#3}}

\newtheorem{df}{Definition}
\newtheorem{prop}[df]{Proposition}
\newtheorem{lema}[df]{Lemma}
\newtheorem{thm}[df]{Theorem}

\newtheorem{problem}[df]{Problem}

\date{\today}

\thanks{The author is supported by the EPSRC}

\title[Heights on hyperelliptic Jacobians]{Computing N\'eron-Tate heights of points on hyperelliptic Jacobians}

\newcounter{nootje}
\begin{document}

\begin{abstract}
It was shown by Faltings (\cite{faltings1984calculus}) and Hriljac (\cite{hriljac1985heights}) that the 
N\'eron-Tate height of a point on the Jacobian of a 
curve can be expressed as the self-intersection of 
a corresponding divisor on a regular model of the curve. We make this explicit and use it to give an algorithm for computing N\'eron-Tate heights on Jacobians of hyperelliptic curves. To demonstrate the practicality of our algorithm, we illustrate it by computing N\'eron-Tate heights on Jacobians of hyperelliptic curves of genus $1 \le g \le 9$.
\end{abstract}

\maketitle

\newcommand{\RR}{\mathfrak{R}}
\newcommand{\variety}[2][]{\textrm{zeros}_{#1}\left(#2\right)}

\section{Introduction}

The problem considered in this paper is that of computing the N\'eron-Tate (or canonical) height of a point on the Jacobian of a curve of genus greater than 2. 
For curves of genus 1 and 2 the existing methods (classical in genus 1, and due to Flynn, Smart, Cassels and others in genus 2 \cite{cassels-and-flynn} and \cite{flynn_and_smart}) make use of explicit equations for projective embeddings of Jacobians,
and have proven to be very successful in practice. It does not seem practical
at present to give explicit equations for Jacobians of curves
of genus 3 and above (see \cite{stubbsthesis} and \cite{mullerthesis} for recent attempts and an examination of the difficuties faced). We propose an alternative approach to
computing the N\'eron-Tate height based on Arakelov theory.
To demonstrate that our method is practical, we give 
numerical examples where we compute heights of points on Jacobians of hyperelliptic
curves of genus $1 \leq g \leq 9$. 

The main application of these computations is at 
present the computation of regulators of 
hyperelliptic Jacobians up to rational squares; 
this will allow the verification of the 
conjectures of Birch and Swinnerton-Dyer up to rational squares. 
The author has also developed an algorithm 
to bound the difference between the na\"ive and N\'eron-Tate heights, 
again using Arakelov theory \cite{holmes2010saturation}. 
Together, the algorithms allow a number of further applications 
such as computing a basis of the Mordell-Weil group 
of a hyperelliptic Jacobian, computing integral points on 
hyperelliptic curves (see \cite{integral_points_on_hyp}) and of 
course verifying the conjectures of Birch and Swinnerton-Dyer, up to the
order of the Shafarevich-Tate group.

Whilst the theoretical sections of this paper are largely independent 
of the curve chosen, the very geometric nature of Arakelov theory 
means that the details of the algorithm, and especially its implementation, 
will depend greatly on the geometry of the curve considered. 
Let $C$ be a curve defined over a number field $k$.
Our method makes a number of computational assumptions:
\begin{enumerate}
\item[(a)] We have a uniform and convenient way of representing divisor classes
on the curve $C$.
\item[(b)] We are able to rigorously compute abelian integrals on $C$ to 
any required (reasonable) precision.
\item[(c)] We are able to write down a regular proper model $\cc$ 
for 
$C$ over the integers $\oo_k$ (though this need not be minimal).
\item[(d)] For each non-Archimedean place $\nu$ that is 
a prime of bad reduction for $\cc$, we are able to
compute the intersection matrix of the special fibre $\cc_\nu$.
\item[(e)] We have a way of computing Riemann-Roch spaces of divisors
on $C$.
\end{enumerate}
Assumptions (a) and (b) cause us to restrict our attention
to hyperelliptic curves. The computer algebra package {\tt MAGMA} \cite{MAGMA}
has in-built commands to deal with all of these for hyperelliptic curves
over number fields. For (a), it uses
Mumford's representation 
(see \cite[3.19]{mumford1984tata}) for divisor classes. A {\tt MAGMA} implementation by P. van Wamelen is used for integral computations in (b); this does not use the usual numerical integration techniques as these are inherently non-rigorous; instead, hyperelliptic functions are locally approximated by truncated power series and formally integrated. 
The computation of the intersection matrices of the special
fibres at the bad places is produced by {\tt MAGMA}
as by-product
of the computation of the regular proper model, implemented by S. Donnelly using techniques as in \cite[Chapter 8]{liu2006algebraic}.
For computing Riemann-Roch spaces, {\tt MAGMA} makes use of the
method of Hess \cite{hess}.

In order to simplify the exposition,
we restrict our attention to curves with a rational Weierstrass point. 
As such, unless otherwise stated, $C$ will denote an 
odd-degree hyperelliptic curve over a number field $k$, 
and $\cc$ a proper regular---though not necessarily minimal---model 
of $C$ over the integers $\oo_k$. 

   The author wishes to thank Samir Siksek for introducing him to this
fascinating problem, as well as for much helpful advice and a careful
reading of this manuscript. Thanks are also due to Martin Bright and
Jan Steffen M\"{u}ller amongst others for very helpful discussions.

\section{A Formula of Faltings and Hriljac}\label{sec:formula}
As before $C$ is an odd degree hyperelliptic curve over a number field
$k$. We fix once and for all the following notation:
\begin{itemize}
\item $M_k$ a proper set of places of $k$;
\item $M_k^0$ the subset of non-Archimedean places of $k$;
\item $M_k^\infty$ the subset of Archimedean places of $k$;
\item $\kappa(\nu)$ the residue field at a $\nu \in M_k^0$; 
\item $\iota\uv$ the usual intersection pairing 
between divisors over $\nu \in M_k^0$ 
(see \cite[IV,\S1]{lang1988introduction}). 
\end{itemize}

We shall make use of the following result which can 
be found in Lang's book \cite[IV, \S2]{lang1988introduction}. 
\begin{thm} (Faltings and Hriljac)\label{thm:basic}
Let $D$ be a degree zero divisor on $C$, and let
$E$ be any divisor linearly equivalent to $D$ but with disjoint
support. 
Then the height with respect to twice the $\vartheta$-divisor of the point on $\Jac(C)$ corresponding to $D$ is given by
\begin{equation*}
 \hat{h}_{2\vartheta}\left(\oo(D)\right) = 
 -\sum_{\nu \in M_k^0}\log\abs{\kappa(\nu)}\intersect{\overline{D}+\Phi(D)}{\overline{E}} - \frac{1}{2}\sum_{\nu\in M_k^\infty}g_{D,\nu}(E)
\end{equation*}
where $\Phi$ and $g_{D,\nu}$ are defined as follows: 

- $\Phi$ sends a divisor on the curve $C$ to an element of the group of fibral $\mathbb{Q}$-divisors on $\cc$ with order zero along the irreducible component containing infinity, such that for any divisor $D$ on $C$ (with Zariski closure $\overline{D}$) and fibral divisor $Y$ on $\cc$, we have $\intersect{\overline{D}+\Phi(D)}{Y}=0$. 

- $g_{D,\nu}$ denotes a Green function for the divisor $D$, when $C$ is viewed 
as a complex manifold via the embedding $\nu$. 
\end{thm}

\bigskip

Our strategy for computing the N\'eron-Tate height of a degree zero divisor $D$
using the above formula is as follows:
\begin{enumerate}
\item Determine a suitable divisor $E$ as above. This is
explained in Section~\ref{sec:step_1}.
\item Determine a finite set $\RR\subset M_k^0$ such that for 
non-Archimedean places $\nu$ \textbf{not} in $\RR$, 
we have $\intersect{\overline{D}+\Phi(D)}{\overline{E}} = 0$. 
This is explained in Section~\ref{sec:step_2}.
\item Determine $\intersect{\Phi(D)}{\overline{E}}$ for $\nu \in \RR$. 
This is explained in Section~\ref{sec:step_3}.
\item Determine $\intersect{\overline{D}}{\overline{E}}$ for $\nu \in \RR$. 
This is explained in Section~\ref{sec:step_4}.
\item Compute the Green function $g_{D,\nu}(E)$ for Archimedean $\nu$. 
This is explained in Section~\ref{sec:step_5}.
\end{enumerate}
In the final section we give a number of worked examples.

By a straightforward Riemann-Roch computation, 
we can write down 
a divisor in Mumford form that is linearly equivalent
to $D$. We replace $D$ by this Mumford divisor.
Thus we may suppose that
$D=D^\prime-d\cdot\infty$ where $d \le g$ and $D^\prime= \variety{a(x),y-b(x)}$ with $a(x)$, $b(x) \in k[x]$ satisfying certain 
conditions as given in \cite[3.19, Proposition 1.2]{mumford1984tata}.

\section{Step 1: Choosing $E$} 
\label{sec:step_1}
If the support of $D^\prime$ does not contain Weierstrass points, 
choose a $\lambda \in k$ such that $a(\lambda)\ne 0$, and set 
\begin{equation}
 E = \inv(D^\prime)-\frac{d}{2}\variety{x-\lambda}, 
\end{equation}
where $\inv$ denotes the hyperelliptic involution. If $d$ is odd, this is not a divisor but a $\qq$-divisor. 
This is unimportant, but if the reader is troubled he or she should multiply $E$ by 2, and then appeal to the 
quadraticity of the height. 

If the support $D^\prime$ does contain Weierstrass points, either:

a) replace $D$ by a positive multiple of itself to avoid this, or

b) add a divisor of order $2$ to $D$ to remove them.

 (a) is simpler to implement, (b) generally faster computationally. 

Now $E \linequiv -D$, and so 
\begin{equation*}
 \hat{h}_{2\vartheta}\left(\oo(D)\right) = 
 \sum_{\nu \in M_k^0}\log\abs{\kappa(\nu)}\intersect{\overline{D}+\Phi(D)}{\overline{E}} + \frac{1}{2}\sum_{\nu\in M_k^\infty}g_{D,\nu}(E).
\end{equation*}
This is seen by viewing the expression on the right hand side as the global N\'eron pairing on divisor classes, which is a quadratic form; since $E$ is linearly equivalent to $-D$ a minus sign results, which cancels with those in Theorem \ref{thm:basic} to yield the above expression.

\section{Step 2: Determining a Suitable $\RR$}
\label{sec:step_2}
We wish to find a finite set $\RR\subset M_k^0$ such that 
\begin{equation}
 \intersect{\overline{D}+\Phi(D)}{\overline{E}}=0
\end{equation}
for all $\nu \notin \RR$.
 To make this as general as possible, we will for 
the moment just assume $C$ is a smooth curve over $k$ in the weighted 
projective space $\p_k(\listdef[0]{a}{n})$. 
Let $C^\prime$ denote its closure in $\p_{\oo_k}(\listdef[0]{a}{n})$. 
Let $Q_1$ denote the set of places of bad reduction for $C^\prime$, 
outside which $C^\prime$ is smooth over $\oo_k$. 

It suffices to solve our problem for prime divisors, 
as we can then obtain results for general $D$ and $E$ easily. 
Let $X$ and $Y$ be prime divisors on $C$, 
and let $d$ be the degree of $Y$. 
Let $\listdef{H}{d+1}$ be a collection of 
weighted integral homogeneous forms of 
degrees $\listdef{e}{d+1}>0$ on $\p_{\oo_k}(\listdef[0]{a}{n})$, 
geometrically integral on the generic fibre and coprime, 
such that for all pairs $i \ne j$, we have on the generic fibre 
that $H_i \cap H_j \cap C = \emptyset$ (we will confuse $H_i$ 
with the hypersurface it defines). 

Let $Q_2$ be the set of $\nu \in M_k^0\setminus Q_1$ such that
\begin{equation*}
  (\overline{H}_i)\uv \cap (\overline{H}_j)\uv \cap C^\prime\uv \ne 
\emptyset.
\end{equation*}
for some $1 \leq i$, $j \leq d+1$.
Note that $\overline{H}_i \cap \overline{H}_j$ is a zero-dimensional scheme, and so is 
easy to compute in practice.

Let $FF(Y)$ denote the function field of $Y$, a finite extension of $k$,
and let $\N_Y: FF(Y) \rightarrow k$ denote the norm map. 
In practice, we can use Gr\"{o}bner bases to find an isomorphism 
$FF(Y) \stackrel{\sim}{\rightarrow} k[t]/\alpha(t)$,
and so can readily compute $\N_Y$. 

Let $Q_3$ be the set of $\nu \in M^0_k\setminus (Q_1 \cup Q_2)$
such that 
\begin{equation*}
 \ord\uv\left(\N_Y\left(H_i^{e_{i+1}}/H_{i+1}^{e_i}\right)\right)\ne0 
\end{equation*}
for some $1 \leq i \leq d$; while $Q_2$ detects common points of intersection of $\overline{H}_i$, $\overline{H}_j$ and $C'$ over $\nu$, $Q_3$ detects when the intersection numbers of $\overline{H}_i^{e_i+1}$ and $\overline{H}_{i+1}^{e_i}$ with $Y$ over $\nu$ are different. 

Let $\listdef{f}{r}$ be integral weighted 
homogeneous equations for $X$ such 
that no $f_i$ vanishes on $Y$, and set $\deg(f_j)=d_j$. 
Finally let $Q_4$ be the set of $\nu \in M^0\setminus (Q_1 \cup Q_2 \cup Q_3)$
such that
\begin{equation*} 
\ord\uv\left(\N_Y\left(f_j^{e_{1}}/H_{1}^{d_j}\right)\right)\ne0 
\end{equation*}
for some $1 \leq j \leq r$. 

\begin{lema} Set 
\[
\RR = Q_1 \cup Q_2 \cup Q_3\cup Q_4.
\] 
If $\nu \notin \RR$ then 
$\intersect{\overline{X}+\Phi(X)}{\overline{Y}} = 0$.
\end{lema}
\begin{proof}
 Outside $Q_1$, $C^\prime$ is smooth over $\oo_k$, 
and hence it is regular and all its fibres are geometrically integral. 
As a result, 
\begin{equation}
 \intersect{\overline{X}+\Phi(X)}{\overline{Y}} = \intersect{\overline{X}}{\overline{Y}} \text{ for } \nu \notin \RR.
\end{equation}
Suppose $\intersect{\overline{X}}{\overline{Y}} \ne 0$, so $(\overline{X})\uv \cap (\overline{Y})\uv \ne \emptyset$. We will show $\nu \in \RR$. 

Recall that $X$ and $Y$ are cycles on $C'$ of relative dimension zero over $\oo_k$, and so their fibres over closed points are cycles of dimension zero.  Observe that, since the $f_j$ are integral, we must have $\variety{f_j} \supset \overline{X}$ for all $j$. Hence there is some $j_0$ such that $f_{j_0}$ vanishes on some irreducible component of (equivalently, closed point in) $(\overline{Y})\uv$ (in fact, this holds for any $j_0$). As a result, $\intersect{\variety[C^\prime]{f_{j_0}}}{\overline{Y}} > 0$, since we assume $\variety[C^\prime]{f_{j_0}}$ and $\overline{Y}$ have disjoint support on the generic fibre. Now suppose $\nu \notin \RR$. Then for all $i$, since $\nu \notin Q_4$, we must have 
\begin{equation*}
\ord\uv\left(\N_Y\left(f_{j_0}^{e_i}/H_i^{d_{j_0}}\right)\right) = 0. 
\end{equation*}
Hence by \cite[III, Lemma 2.4, p56]{lang1988introduction}, we see that for all $i$, 
\begin{equation*}
0=\intersect{\divisor_{C^\prime}\left( \frac{f_{j_0}^{e_i}}{H_i^{d_{j_0}}}\right)}{\overline{Y}} = e_i \cdot \intersect{\variety[C^\prime]{f_{j_o}}}{\overline{Y}} - d_{j_0}\cdot \intersect{\variety[C^\prime]{H_i}}{\overline{Y}}. 
\end{equation*}

Now as $e_i>0$ and $d_{j_0} >0$, we see that for all $i$, 
\begin{equation}\label{eqn:con}
  \intersect{\variety[C^\prime]{H_i}}{\overline{Y}} >0,
\end{equation}
so every $\variety[C^\prime]{H_i}$ meets $\overline{Y}$. But the zero-dimensional cycles $\variety[C^\prime]{H_i} \cap \overline{Y}$ are pairwise disjoint since $\nu \notin Q_2$, and $\deg(Y)=\deg((\overline{Y})\uv) = d$. Moreover, $\nu \notin Q_3$ shows that the $(d+1)$ cycles $\variety{H_i} \cap \overline{Y}$ are disjoint, and so cannot all meet the zero-dimensional cycle $(\overline{Y})\uv$ as it has degree $d$; this contradicts Equation \ref{eqn:con}, and so we are done. 
\end{proof}
Note that if $a_0=\ldots = a_n = 1$, so $C$ lives in `unweighted' projective space, then the $H_i$ should all be taken to be hyperplanes and the argument simplifies somewhat. 

\section{Step 3: Determining $\intersect{\Phi(D)}{\overline{E}}$}
\label{sec:step_3}

We next discuss the computation of the term $\intersect{\Phi(D)}{\overline{E}}$
for a non-Archimedean place $\nu$. Recall that by our assumptions in
Section~\ref{sec:formula} we are able to write down a proper
regular model $\cc$ and the intersection matrix for $\cc\uv$ for
all bad places $\nu$.
 Clearly $\intersect{\Phi(D)}{\overline{E}}$
vanishes if $\cc\uv$ is integral, in particular if $\nu$ is a good prime. 
Suppose $\nu$ is a bad prime. Since we have the intersection matrix
of $\cc\uv$,
the problem is reduces to the following:

\begin{problem}
 Given a finite place $\nu$, a horizontal divisor $X$ and a prime fibral divisor $Y$ over $\nu$, compute $\intersect{X}{Y}$. 
\end{problem}

We may replace the base space $S=\Spec(\oo_k)$ by its completion $\hat{S}$ at $\nu$ (\cite[III, Proposition 4.4, page 65]{lang1988introduction}), 
and we may further assume that $X$ is 
a prime horizontal divisor on $\cc \times_S \hat{S}$. By Lemma~\ref{lem:h}
 below, this means that the support of $X\uv$ is a closed point of $\cc\uv$, and so we can find an affine open neighbourhood $U = \Spec(A)$ of $X\uv$ in $\cc\times_S \hat{S}$. 

\begin{lema}\label{lem:h}
 If $X$ is a prime horizontal divisor on an arithmetic surface $\mathscr{X}$ over a p-adic local ring $(\oo,\nu)$, then the support of $X\uv$ is a prime divisor on $\mathscr{X}\uv$ (in other words, $X\uv$ is irreducible but not necessarily reduced). 
\end{lema}
\begin{proof}
There exists a number field $L$ and an order $R$ in $L$ such that $X$ 
is isomorphic to $\Spec(R)$. Write $L = k[t]/\alpha(t)$, 
where $\alpha$ monic and irreducible with integral coefficients. 
Let $\kappa$ denote the residue field of $k$, and $\overline{\alpha}$ the image of $\alpha$ in $\kappa[t]$. If $X\uv$ is not irreducible, then there exist $f$, $g \in \kappa[t]$ coprime monic polynomials such that $f \cdot g = \overline{\alpha}$. This factorization of $\overline{\alpha}$ lifts to a factorization of $\alpha$ by Hensel's Lemma. 
\end{proof}

Now it is easy to check 
whether $X\uv \cap Y = \emptyset$; if so, $\intersect{X}{Y}=0$. Further, if $X\uv \subset Y$ and $X\uv$ is not contained in any other fibral prime divisor, then $\intersect{X}{Y}=\deg(X)$; this is easily seen since locally $Y=\variety[U]{\nu}$, and we can take the norm of $\nu$ from the field of fractions $FF(X)$ down to $k$. 

We are left with the case where $X\uv$ lies at the intersection of several fibral prime divisors. We find equations $\overline{f}_1,\ldots,\overline{f}_r \in A \otimes_{\oo_k} \kappa(\nu)$ for $Y$ as a subscheme of $U\uv$. Then choose any lifts $f_i$ of $\overline{f}_i$ to $A$. Now we need two easy results in commutative algebra:

\begin{lema}\label{lema:lifting_gens}
 let $R$ be a ring, $p \in R$ any element, and $I$ an ideal containing $p$. Suppose we have $\listdef{t}{r} \in R$ such that the images $\overline{t}_1, \ldots, \overline{t}_r$ in $R/(p)$ generate the image of $I$ in $R/(p)$. Then $I=(\listdef{t}{r},p)$. 
\end{lema}
\begin{proof}
Let $x \in I$. Write $\overline{x}$ for the image of $x$ in $R/(p)$, and write $\overline{x}=\sum_{i=1}^r \overline{\alpha}_i\overline{t}_i$ for some $\overline{\alpha} \in R/(p)$. Choose lifts $\alpha_i$ of $\overline{\alpha}_i$ to $R$. Then $y \defeq x-\sum_{i=1}^r\alpha_it_i$ has the property that $y \in p\cdot R$. Hence $x$ is in $(t_i,\ldots,t_r,p)$ and so $I \subset (t_1,\ldots,t_r,p)$. Now $p \in I$ by assumption, and $\overline{t}_i \in I/(p)$, so there exists $g_i$ in $I$ such that $g_i - t_i \in p \cdot R$, so $t_i \in I$. 
\end{proof}

\begin{lema}\label{lema:finding_generator}
 Let $R$ be a regular local ring, and $\listdef{t}{r} \in R$ be such that $I \defeq (\listdef{t}{r})$ is a prime ideal of height 1. Now $I$ is principal; write $I = (t)$. Then there exists an index $i$ and a unit $u\in R$ such that $t_i = tu$. In particular, there exists an index $i$ with $I=(t_i)$.  
\end{lema}
\begin{proof}
$R$ is a unique factorization domain, and so f
For each $i$ we can write 
$t_i = t^\prime_i t$ for some $t_i^\prime\in R$. Hence $I=t\cdot(\listdef{t^\prime}{r})$, and $(\listdef{t^\prime}{r})=1$. We want to show some $t^\prime_i$ 
is a unit. Suppose not; then since $A$ is local, all the $t_i^\prime$ 
lie in the maximal ideal, so $(\listdef{t^\prime}{r})$ is contained 
in the maximal ideal, a contradiction.
\end{proof}
Now from these we see that one of the $f_i$ or $\nu$ must be an 
equation for $Y$ in a neighbourhood of $X\uv$ (and it cannot be $\nu$ as $X\uv$ lies on an intersection of fibral primes). Now if any $f_i$ vanishes on $X\uv$, 
it cannot be the $f_i$ we seek. Exclude such $f_i$, and then for each of 
the remaining $f_i$ compute its norm from $FF(X)$ to the completion of $k$. 
The minimum of the valuations of such norms will be achieved by any 
$f_i$ which is a local equation for $Y$ at $X\uv$, 
and hence $\intersect{Y}{X}$ is equal to the minimum of the 
valuations of the norms.

\section{Step 4: Determining $\intersect{\overline{D}}{\overline{E}}$}
\label{sec:step_4}

Finally, we come to what appears to be the meat of the problem 
for non-Archimedean places: given two horizontal divisors $D$ and $E$ and a place $\nu \in \mathfrak{R}$, 
compute the intersection $\intersect{\overline{D}}{\overline{E}}$. However, the techniques used in previous sections  actually make this very simple. 

Fix a non-Archimedean place $\nu$. Let $\hat{S}$ denote the $\nu$-adic completion of $S$, and set $\hat{\cc} = \cc \times_S \hat{S}$. It is sufficient to compute the intersection $\intersect{X}{Y}$ where $X$ and $Y$ are prime horizontal divisors on $\hat{\cc}$; in particular (by Lemma \ref{lem:h}), the supports of $X\uv$ and $Y\uv$ are closed points of $\cc\uv = \hat{\cc}\uv$ 

Now if $\supp(X\uv) \ne \supp(Y\uv)$, then $\intersect{X}{Y}=0$. Otherwise, let $U=\Spec(A)$ be an affine open neighbourhood of $\supp(X\uv)$.
 Let $\listdef{f}{r}$ generate the ideal of $X$ on $U$; 
then by Lemma \ref{lema:finding_generator} we know that 
some $f_i$ generates the ideal of $X$ in a neighbourhood of $X\uv$. If $f_j$ vanishes on $Y$, we can throw it away. We obtain
\begin{prop}\label{prop:norms_1}
\begin{displaymath}
 \intersect{X}{Y}=\min_i \left(\ord\uv\left(f_i[Y]\right)\right)
\end{displaymath}
as $i$ runs over $\{1,\ldots,r\}$ such that $f_i$ does not vanish identically on $Y\uv$. Here $f[Y]$ is defined to be either
\begin{enumerate}
\item the norm from $FF(Y)$ to the 
completion of $k$ of the image of $f_i$ in $FF(Y)$, 
or, equivalently,
\item
$\prod_jf_i(p_j)^{n_j}$ where $Y=\sum_jn_jp_j$ over some finite extension $l/k$ (see \cite[II,\S 2, page 57]{lang1988introduction}).
\end{enumerate}
\end{prop}
\begin{proof}
If $f_i$ is not identically zero on $Y\uv$, then $\variety[\hat{\cc}]{f_i}$ and $Y$ have no common component and moreover $f_i$ is regular on a neighbourhood of $Y$ so $\intersect{\text{poles}_{C'}(f_i)}{Y}=0$, and so \cite[II, Lemma 2.4, p56]{lang1988introduction} shows that
\begin{equation}
 \intersect{\variety[\hat{\cc}]{f_i}}{Y}=\ord\uv\left(f_i[Y]\right). 
\end{equation}

Now $\variety[\hat{\cc}]{f_i} \ge X$, so $\ord\uv\left(f_i[Y]\right)\ge \intersect{X}{Y}$. Moreover, by lemma \ref{lema:lifting_gens} there is an index $i_0$ such that $f_{i_0}$ generates $X$ near $X\uv$, and since $X\uv=Y\uv$ is irreducible we have that 
\begin{equation}
 \begin{split}
  \intersect{X}{Y} = \intersect{\variety[\hat{\cc}]{f_{i_0}}}{Y} & = \sum_{p} \text{length}_{\oo_p}\left(\frac{\oo_p}{f_{i_0},I_Y}\right)\\
 & = \text{length}_{\oo_{X\uv}}\left(\frac{\oo_{X\uv}}{f_{i_0},I_Y}\right)
 \end{split}
\end{equation}
where the sum is over closed points $p$ of $\hat{\cc}$ lying over $\nu$, and $I_Y$ is the defining ideal for $Y$ in the local ring under consideration. Now any other $f_i$ will have 
\begin{equation}
\intersect{\variety[\hat{\cc}]{f_i}}{Y}\ge \intersect{\variety[\hat{\cc}]{f_{i_0}}}{Y}, 
\end{equation}
so the result follows. 
\end{proof}

As regards the computation of the $f_i[Y]$, definitions (1) and (2) given in Proposition \ref{prop:norms_1} lead to slightly different approaches, 
but both make use of Pauli's algorithms \cite{pauli2001computation}. 
In our implementation, discussed in Section~\ref{sec:examples},
we use (2)
as it seems easier; however (1) may lead to an implementation
that is faster in practice.

\section{Step 5: Computing $g_{D,\nu}(E)$}
\label{sec:step_5}

Finally, we must compute the Archimedean contribution. Fix for the remainder of this section an embedding $\sigma$ of $k$ in $\ccc$ corresponding to a place $\nu \in M_k^\infty$. Let $C_\sigma$ denote the Riemann surface corresponding to $C \times_{k,\sigma} \ccc$.

\subsection{The PDE to be Solved}
\label{The pde we need to solve}
As a starting point, we take \cite[Chapter 13, Theorem 7.2]{lang1983fundamentals}, which we summarise here.

Given a divisor $a$ on $C_{\sigma}$ of degree zero, let $\omega$ be a differential form on $C_{\sigma}$ such that the residue divisor $\res(\omega)$ equals $a$ (such an $\omega$ can always be found using the Riemann-Roch Theorem). Normalise $\omega$ by adding on holomorphic forms until the periods of $\omega$ are purely imaginary. Let 
\begin{equation}
d g_a \defeq \omega + \bar{\omega}.
	\label{pde1}
\end{equation}
Then $g_a$ is a Green function for $a$. 
Thus it remains to find, normalise and integrate such a form $\omega$.

\subsection{Application of theta functions to the function theory of hyperelliptic curves}

We can use $\vartheta$-functions to solve the partial
differential equation \eqref{pde1} of Section \ref{The pde we need to solve}, in a very simple way. For background on $\vartheta$-functions we refer to the first two books of the `Tata lectures on theta' trilogy, \cite{mumford1983tata}, \cite{mumford1984tata}. $\vartheta$-functions are complex analytic functions on $\mathbb{C}^g$ which satisfy some quasi-periodicity conditions, thus they are an excellent source of differential forms on the (analytic) Jacobian of $C_{\sigma}$. To get from this a differential form on $C_{\sigma}$ we simply use that $C_{\sigma}$ is canonically embedded in $\Jac(C_{\sigma})$ by the Abel-Jacobi map, so we can pull back forms from $\Jac(C_{\sigma})$ to $C_{\sigma}$.


Fix a symplectic homology basis $A_i,B_i$ on $C_{\sigma}$ as in \cite{mumford1984tata}; by this we mean that if $i(-,-)$ denotes the intersection of paths, then we require that the $A_i,B_i$ form a basis of $H_1(C_{\sigma},\mathbb{Z})$ such that
\[i(A_i,A_j) = i(B_i,B_j) = 0 \; \text{for} \; i \neq j\]
and
\[i(A_i,B_j) = \delta_{ij}.\]
We also choose a basis $\omega_1,\ldots \omega_g$ of holomorphic 
1-forms on $C_{\sigma}$, normalised such that 
\[\int_{A_i} \omega_j = \delta_{ij}.\]

We recall the definition and some basic properties of the multivariate $\vartheta$-function:

\begin{equation}
\label{theta}
\vartheta(z;\Omega)\defeq \sum_{\underline{n} \: \text{in} \: \mathbb{Z}^g} \exp (\pi i \underline{n} \Omega \underline{n}^\textsc{T} + 2 \pi i \underline{n}\cdot z)  
\end{equation}
which converges  
 for $z$ in $\mathbb{C}^g$ and $\Omega$ a $g \times g$ symmetric complex matrix with positive definite imaginary part. The $\vartheta$-function satisfies the following periodicity conditions for $\underline{m}, \underline{n}$ in $\mathbb{Z}^g$:
\begin{equation}
\label{z periodicity}
\vartheta(z+\underline{m}; \Omega) = \vartheta(z; \Omega),
\end{equation}
\begin{equation}
\label{t periodicity}
    \vartheta(z+\underline{n}\Omega;\Omega) = \exp(-\pi i \underline{n} \Omega \underline{n}^{\textsc{T}} -2 \pi i \underline{n} z)\,\vartheta(z;\Omega) .
\end{equation}

We will set $\Omega$ to be the period matrix of the analytic Jacobian of $C_{\sigma}$ with respect to the fixed symplectic homology basis (as in  \cite{mumford1984tata}), and $z$ will be a coordinate on the analytic Jacobian. This means that 
\[ \Omega_{ij} = \int_{B_i} \omega_j .\]
Let 
\[\delta^\prime \defeq \left(\frac{1}{2},\frac{1}{2},\ldots ,\frac{1}{2},\frac{1}{2}\right) \in \frac{1}{2}\mathbb{Z}^g\]

\[\delta^{\prime\prime}\defeq\left(\frac{g}{2},\frac{g-1}{2},\ldots ,1,\frac{1}{2}\right) \in \frac{1}{2}\mathbb{Z}^g\]
\[ \Delta \defeq \Omega \cdot\delta^\prime + \delta^{\prime\prime}.\]

Then \cite[Theorem 5.3, part 1]{mumford1984tata} tells us that 
$\vartheta(\Delta - z)=0$ if and only if
there are $P_1,\ldots P_{g-1}$ in $C_{\sigma}$ such that 
\begin{equation*}
z \equiv \sum_{i=1}^{g-1} \int_\infty^{P_i} \underline{\omega} \pmod{\mathbb{Z}^g+\Omega\mathbb{Z}^g}.
 \end{equation*}
This is a crucial result which allows us to construct a quasifunction on $\Jac(C_{\sigma})$ with prescribed zeros, and from this obtain the Green function we seek.

\subsection{Solution of the Partial Differential Equation}
Let $D$, $D_0$ be two effective reduced divisors of degree $g$ on $C_{\sigma}$ with disjoint support, containing no Weierstrass points or points at infinity, nor any pairs $p+q$ of points such that $p=\inv(q)$. Then the classes $[\oo(D-g \cdot \infty)]$ and $[\oo(D_0-g \cdot \infty)]$ lie outside 
the $\vartheta$-divisor on the Jacobian; indeed, the association 
$D \mapsto [\oo(D-g\cdot \infty)]$ is an isomorphism from divisors 
with the above properties to $\Jac(C_\sigma)\setminus\vartheta$, see \cite[3.31]{mumford1984tata}. Write $\alpha:\Div(C_{\sigma}) \rightarrow \Jac(C_{\sigma})$ for the map sending a divisor $E$ to the class $[\oo(E-\deg(E)\cdot\infty)]$. 

For $z$ in $\Jac(C_{\sigma})$ we set 
\[G(z) = \frac{\vartheta(z+\Delta - \alpha(D))}{\vartheta(z+\Delta - \alpha(D_0))}.\]
Then for $p$ in $C_{\sigma}$ we set $F(p) = G(\alpha(p))$ so
\begin{equation}
\label{F}
F(p) = \frac{\vartheta(\alpha(p)+\Delta - \alpha(D))}{\vartheta(\alpha(p)+\Delta - \alpha(D_0))}.
\end{equation}

If we let $\omega = d \log F(p)$ then it is clear that $\res(\omega) = D-D_0$. It then remains to normalise $\omega$ to make its periods purely imaginary, and then integrate it. We have a homology basis $A_i,B_i$, and we find:
\[
\int_{A_k} \omega = \int_{A_k} d \log F(p) = \log G(\alpha(p)+e_k) - \log G(\alpha(p)) = 0\]
(where $e_k = (0,0,\ldots 0,1,0 \ldots ,0)$ with the $1$ being in the
$k$-th position), and

\begin{equation*}
\begin{split}
\int_{B_k} \omega = \int_{B_k} d \log F(p) = \log G(\alpha(p)+\Omega.e_k) - \log G(\alpha(p)) \\
= 2\pi i \, e_k ^T \cdot (\alpha(D)-\alpha(D_0)).
\end{split}
\end{equation*}
From this we can deduce that the normalisation is

\begin{equation*}
\begin{split}
\omega = d & \log \left[ \frac{\vartheta(\alpha(p)+\Delta - \alpha(D))}{\vartheta(\alpha(p)+\Delta - \alpha(D_0))} \right]   \\
& - 2 \pi i \left[ (\im( \Omega))^{-1} \im ( \alpha(D)-\alpha(D_0)) \right].\left[\begin{array}{c} \omega_1 \\ \omega_2 \\ \vdots \\ \omega_g \end{array} \right]
\end{split}
\end{equation*}
where $p$ is a point on $C_{\sigma}$.

Now we integrate to get the Green function $g_{D-D_0}(p) = \int_{\infty_{C_{\sigma}}}^p \omega + \overline{\omega}$, where $\infty_{C_\sigma}$ denotes the point at infinity on $C_\sigma$:

\begin{equation*}
\begin{split}
g_{D-D_0}(p) =  2 \log & \left|\frac{\vartheta(\alpha(p)+\Delta - \alpha(D))}{\vartheta(\alpha(p)+\Delta - \alpha(D_0))} \right| 
\\  \\ &+ 4 \pi  \left[ (\im( \Omega))^{-1} \im ( \alpha(D)-\alpha(D_0)) \right].\im \left( \int_{\infty_{C_{\sigma}}}^p\left[\begin{array}{c} \omega_1 \\ \omega_2 \\ \vdots \\ \omega_g \end{array} \right] \right)
\\ \\
= 2 \log & \left|\frac{\vartheta(\alpha(p)+\Delta - \alpha(D))}{\vartheta(\alpha(p)+\Delta - \alpha(D_0))} \right| \\
 &+ 4 \pi   (\im (\Omega))^{-1}  \cdot \im ( \alpha(D)-\alpha(D_0)) \cdot \im \left( \alpha(p) \right).
\end{split}
\end{equation*}

Given divisors $D$, $D_0$ and $E$, $E_0$ containing no Weierstrass 
points or infinite points or pairs of 
points which are involutions of each other, and having disjoint support 
we can this formula to compute 
$\frac{1}{2}g_{D-D_0} [ E-E_0]$ which is 
simply the product over points $p \in \supp(E-E_0)$ 
of the complex absolute value of $g(p)$. We are done.

\section{Examples}
\label{sec:examples}

We have created a test implementation of the above algorithm in {\tt MAGMA}. 
The following results were obtained using a 2.50 GHz Intel Core2 Quad  CPU  Q9300:

First, we let $C/\mathbb{Q}$ be the genus $3$ hyperelliptic curve given by
\[
C: y^2 = x^7-15x^3+11x^2-13x+25.
\]
Let $D$, $E$ be the points on the Jacobian corresponding to the 
degree $0$ divisors
$(1,3)-\infty$, $E=(0,-5)-\infty$ respectively.
We obtain the following (writing $\hat{h} = \hat h_{2 \vartheta}$):
\[\hat{h}(D) = 1.77668 \ldots\]
\[\hat{h}(E) = 1.94307 \ldots\]
\[\hat{h}(D+E) = 4.35844 \ldots\]
\[\hat{h}(D-E) = 3.08107 \ldots\]
\[2\hat{h}(D)+2\hat{h}(E) - \hat{h}(D+E) - \hat{h}(D-E) = 1.26217 \times 10^{-28}\]
with a total running time of 31.75 seconds.
We note that our result is consistent with the parallelogram law for the
N\'eron-Tate height, which provides a useful check that our implementation
is running correctly.

Next we give two families of curves of increasing genus. 
Firstly the family $y^2 = x^{2g+1}+2x^2-10x+11$ with $D$ 
denoting the point on the Jacobian corresponding to the degree $0$
divisor $(1,2) - \infty$ (all times are in seconds unless otherwise stated):

\begin{tabular}{l|c|r}
$g$&$\hat{h}(D)$ & time\\ \hline
1 & $1.11466 \ldots$ & 1.94\\
2 & $1.35816 \ldots$ & 6.44\\
3 & $1.50616 \ldots$ & 15.10\\
4 & $1.61569 \ldots$ & 32.71\\
5 & $63.4292 \ldots$ & 72.23\\
6 & $1.77778 \ldots$ & 212.37\\
7 & $51.0115 \ldots$ & 20 minutes\\
8 & $1.89845 \ldots$ & 3 hours\\
9 & $78.8561 \ldots$ & 16 hours\\
\end{tabular}

Now we consider the family $y^2 = x^{2g+1}+6x^2-4x+1$ with $D$ denoting the point $(1,2)+(0,1) - 2\cdot \infty$ on the Jacobian:

\begin{tabular}{l|c|r}
$g$&$\hat{h}(D)$ & time/seconds\\ \hline
1 & $1.41617 \ldots$ & 2.06\\
2 & $1.37403 \ldots$ & 6.73\\
3 & $1.50396 \ldots$ & 15.62\\
4 & $1.40959 \ldots$ & 32.60\\
5 & $1.70191 \ldots$ & 76.48\\
6 & $1.81093 \ldots$ & 291.17\\
7 & $1.71980 \ldots$ & 1621.50\\
\end{tabular}

A fully-functioning and more efficient implementation of N\'eron-Tate 
height computations is currently being carried out in {\tt MAGMA} 
by J. S. M\"{u}ller, 
combining ideas from this paper with those from his own PhD 
thesis, \cite{mullerthesis}. 
M\"{u}ller's approach to computing $\intersect{D}{E}$ is quite different 
from that used here; he uses Gr\"obner bases to compute directly the 
$\oo_p$-length of the modules $\frac{\oo_p}{I_D+I_E}$ as $p$ 
runs over closed points of the special fibre, in contrast 
to the method in this paper where we compute norms down to the ground field and then compute valuations.

\bibliographystyle{alpha} 
\bibliography{../../bibtex/prebib.bib}

\bigskip

\end{document}